\documentclass[11pt]{amsart}

\usepackage{latexsym,amssymb,amsfonts,amsmath,mathrsfs,bm}
\newcommand{\Phihat}{\widehat{\Phi}}
\newcommand{\sumstar}{\sideset{}{^{*}}\sum}
\newcommand{\chibar}{\overline{\chi}}

\newcommand{\R}{{\mathbb R}}
\newcommand{\C}{{\mathbb C}}
\newcommand{\eps}{{\varepsilon}}

\newtheorem{theorem}{Theorem}[section]
\newtheorem{theorem1}{Theorem}

\newtheorem*{corollary1}{Corollary A}
\newtheorem*{proposition1}{Proposition A}
\newtheorem{lemma}[theorem]{Lemma}
\newtheorem*{lemma1}{Lemma 5.1*}
\theoremstyle{definition}

\theoremstyle{remark}

\newtheorem*{remark1}{Remark}

\numberwithin{equation}{section}

\begin{document}

\title[Critical zeros of Dirichlet $L$-functions]{Critical zeros of Dirichlet $L$-functions}
\author{J.B. Conrey, H. Iwaniec and K. Soundararajan}
\begin{address}{American Institute of Mathematics and University of Bristol}
\end{address}
\begin{address}{Rutgers University}
\end{address}
\begin{address}{Stanford University}
\end{address}


\begin{abstract}
We use the Asymptotic Large Sieve and Levinson's method to obtain lower bounds for the proportion of simple zeros on the critical line of 
the twists by primitive Dirichlet characters of a fixed L-function of degree 1,2, or 3.
\end{abstract}

\maketitle

\section{Introduction}\label{section1}

In this paper we prove that at least $56\%$ of zeros of the family of Dirichlet $L$-functions
\begin{equation}\label{11}
L(s,\chi) = \sum_{1}^{\infty}\chi(n)n^{-s}
\end{equation}
are on the line $\Re s = 1/2$.  Therefore one may say that the Riemann Hypothesis for this family
is more likely to be true than not!

We are going to qualify this statement in asymptotic terms.  Let $\chi\pmod{q}$ be a primitive character.
The total number of zeros  $\rho = \beta + i\gamma$ of $L(s, \chi)$ with $0 < \beta < 1$ and $|\gamma| \leq T$,
say $N(T,\chi)$, is known asymptotically very precisely
\begin{equation}\label{12}
N(T,\chi) = \frac{T}{\pi}\log{\frac{qT}{2\pi e}} + O(\log{qT}), \quad T \geq 3.
\end{equation}
The number of these zeros with $\beta = 1/2$, say $N_{0}(T,\chi)$, is known to satisfy
\begin{equation}\label{13}
N_{0}(T,\chi) \gg N(T,\chi)
\end{equation}
provided $q$ is fixed and $T$ is sufficiently large in terms of $q$, where the implied constant is absolute.
This result in the case of the Riemann zeta function ($q=1$) is due to A. Selberg \cite{S}.
Selberg's method does not produce a considerable proportion of the critical zeros, contrary to the other method of
N. Levinson \cite{L} which can yield a respectful number of at least $34\%$.  In a series of works by B. Conrey
\cite{C1}, \cite{C2} and with others \cite{BCY} the method of Levinson has been explained conceptually,
clarified technically and substantially refined by means of new devices, leading to the current record of over
$41\%$ of the critical zeros (well, only for the Riemann zeta function, however the case of $L(s,\chi)$ is
not much different).  We shall follow the ideas of \cite{C1} and adapt its technology (such as handling the
approximate functional equation) to our needs.

There are two aspects when counting the zeros of $L(s,\chi)$; the $t$-aspect and the $q$-aspect, however we shall
focus only on the latter.  Actually we do perform a hybrid aspect, but we down-size its $t$-component because
our arguments do not benefit in this regard.

Using our construction we actually count the simple zeros (see Appendix).  Denote by $N_{0}'(T,\chi)$ the
number of simple zeros of $L(s,\chi)$, $\rho = 1/2 + i\gamma$ with $|\gamma| \leq T$, so
$N_{0}(T,\chi) \geq N_{0}'(T,\chi)$.

Let $\Psi(x)$ be a non-negative function, smooth, compactly supported on $\R^{+}$.  Put
\begin{equation}\label{14}
\mathcal{N}(T,Q) = \sum_{q}\frac{\Psi(q/Q)}{\varphi(q)}
\sumstar_{\chi (\bmod{q})}N(T,\chi)
\end{equation}
where $Q \geq 3$ and $T \geq 3$.  Here the superscript $*$ restricts the summation to the primitive characters.
Let $\mathcal{N}_{0}'(T,Q)$ denote the same sum, but with $N(T,\chi)$ replaced by $N_{0}'(T,\chi)$.

\begin{theorem1}\label{theorem1}
For $Q$ and $T$ with $(\log{Q})^{6} \leq T \leq (\log{Q})^{A}$ we have
\begin{equation}\label{15}
\mathcal{N}_{0}'(T,Q) \geq \frac{14}{25}\mathcal{N}(T,Q),
\end{equation}
where $A \geq 6$ is any constant, provided $Q$ is sufficiently large in terms of $A$.
\end{theorem1}

In some sense Levinson's approach to counting critical zeros starts from the opposite direction to that of Selberg.
Indeed, Selberg adds a zero between sign changes of a real function (a safe route), while Levinson subtracts unwanted
zeros from a total collection (a risk of getting negative outcome).  We shall give a sketch of Levinson's method in the
Appendix.  His approach begins with taking a suitable linear combination of $\zeta(s)$ and its derivative.
Likewise we take
\begin{equation}\label{16}
G(s,\chi) = L(s,\chi) + \lambda L'(s,\chi)
\end{equation}
with $\lambda = 1/r\log{\textbf{q}}$, where $r$ is a positive constant and for notational convenience we put
\begin{equation}\label{17}
\textbf{q} = q/\pi.
\end{equation}
This idea to take a more general linear combination of higher order derivatives has been fully developed in \cite{C1},
\cite{CG}.  However, by taking only $L(s,\chi)$ and $L'(s,\chi)$ we shall be also able to derive a lower bound for the
percentage of simple zeros.

At some point, after applying Littlewood's formula, one needs an upper bound for the integral
\begin{equation}\label{18}
\frac{1}{2T}\int_{-T}^{T}|G(\sigma + it,\chi)|^{2}dt
\end{equation}
over the vertical segment with $\sigma < 1/2$, $\sigma$ near $1/2$.  But such a straightforward treatment does not work,
because the extreme values of $G(s,\chi)$ make the second power moment \eqref{18} rather large.  These extreme large
values appear rarely, nevertheless they need to be mollified.  To this end (an idea first used by Selberg) we attach
to $G(s,\chi)$ a mollifying factor $M(s,\chi)$ before embarking to Littlewood's formula.  An experience shows that
a good choice is given by
\begin{equation}\label{19}
M(s,\chi) = \sum_{m \leq X}\mu(m)\chi(m)m^{-s}P(1- \frac{\log{m}}{\log{X}})
\end{equation}
where $P(x)$ is a smooth function with $P(0) = 0$ and $P(1) = 1$.  One may think of $M(s,\chi)$ as an approximation
to $1/G(s,\chi)$, however this view point must be considered with some reservation.  A comprehensive study of mollifiers
can be found in the survey articles and papers by the first author \cite{C1},\cite{C2},\cite{C3}.

Having said that, we are led to consider integrals of type
\begin{equation}\label{110}
I_{\chi} = \int|G(\sigma + it,\chi)M(\tfrac{1}{2} + it,\chi)|^{2}\Phi(t)dt
\end{equation}
instead of \eqref{18}.  Here we have also introduced a factor $\Phi(t)$ not for dampening large values of $G(s,\chi)$,
but exclusively for smoothing out the integration.  We assume that $\Phi(t)$ is smooth, $\Phi(t) \geq 0$ with
\begin{equation}\label{111}
\Phihat(1) = \int_{-\infty}^{\infty}\Phi(t)dt > 0
\end{equation}
and
\begin{equation}\label{112}
(1+|t|)^{j}\Phi^{(j)}(t) \ll \Big(1 + \frac{|t|}{T}\Big)^{-A}
\end{equation}
for any $j \geq 0$ and any $A \geq 0$, the implied constant depending on $j$ and $A$.  If desired, this smoothing factor
in \eqref{110} can be easily replaced by the sharp cut $|t| < T$ by exploiting the positivity features.  In this case
$\Phihat(1) = 2T$ while in general we think of having $\Phi(t)$ with $\Phihat(1) \asymp T$.

Due to the effect of mollification we expect that
\begin{equation}\label{113}
I_{\chi} \sim c~\Phihat(1) \qquad \mbox{ as } \, q \rightarrow \infty,
\end{equation}
if $(\log{q})^{6} \leq T \leq (\log{q})^{A}$ and $X = q^{\theta}$, $0 < \theta < 1$.  Here $c$ is a positive
constant which depends on the function $P$ in \eqref{19}, and it is the same one for every $\chi \pmod{q}$.
If there were a perfect mollifier one would guess that \eqref{113} holds with $c=1$, but it is not
going to happen.  Definitely $c > 1$.  Note that the $L$-functions in $G(s,\chi)$ run in \eqref{110} over the line
$\sigma < 1/2$, whereas the mollifier $M(s, \chi)$ appears on the critical line.  We leave this little mystery
for the reader's attention and contemplation.

Recall that the mollifier \eqref{19} is a Dirichlet polynomial of length $X = q^{\theta}$.  Naturally, the larger
$\theta$ is admitted the better mollification can be achieved, resulting in smaller value of $c$ in \eqref{113}, which
is our goal.  At the present state of technology we are unable to prove \eqref{113} for individual characters
$\chi \pmod{q}$ even for very short mollifiers.  However, by averaging over the characters we are able to get
\begin{equation}\label{114}
\frac{1}{\varphi^{*}(q)}\sumstar_{\chi (\bmod{q})}I_{\chi} \sim c~\Phihat(1)
\end{equation}
where $\varphi^{*}(q)$ denotes the number of primitive characters, $\varphi^{*} = \mu * \varphi$ (assume
$q \not\equiv 2 \pmod{4}$, or else $\varphi^{*}(q) = 0$).  With the amount of averaging $\varphi^{*}(q)$ we can
accept mollifiers of length $X = q^{1/2 - \eps}$.  The situation looks very much the same as in the Levinson
work on $\zeta(s)$ in the $t$-aspect, thus one can derive the analogous result in the $q$-aspect;
\begin{equation}\label{115}
\sumstar_{\chi (\bmod{q})}N_{0}(T,\chi) > 0.34\sumstar_{\chi (\bmod{q})}N(T,\chi).
\end{equation}
One can pursuit further along the lines of \cite{C1} allowing the mollifier of length $X = q^{4/7 - \eps}$ and
getting \eqref{115} with $0.34$ increased to $0.4$.

In this paper we introduce further averaging over the conductor $q$ getting a larger improvement as in Theorem
\ref{theorem1}.  This improvement comes from the fact that our mollifier has length $X = q^{1-\eps}$.
It seems we reached the limit of the mollification technology, with respect to the length, because it is unlikely
that a mollifier longer than the size of the conductor can be worked out unconditionally without recourse to the
Riemann Hypothesis.  Of course, some small improvements over $0.56$ are possible by shaping a bit $G(s,\chi)$
and $M(s,\chi)$.

For simplicity in this paper we take \eqref{19} with $P(x) = x$, that is
\begin{equation}\label{116}
M(s,\chi) = \sum_{m \leq X}\mu(m)\chi(m)m^{-s}\Big(1 - \frac{\log{m}}{\log{X}}\Big).
\end{equation}
In this special case we prove

\begin{theorem1}\label{theorem2}
Let $X = \textbf{q}^{\theta}$ with $0 < \theta < 1$, $\lambda = 1/r\log{\textbf{q}}$ with $r>0$ and $\sigma = \frac{1}{2} -
\frac{R}{\log{\textbf{q}}}$ with $R>0$.  Then
\begin{equation}\label{117}
\sum_{q}\frac{\Psi(q/Q)}{\varphi(q)}\sumstar_{\chi (\bmod{q})}I_{\chi} \sim
c(\theta, r, R)\Phihat(1)\sum_{q}\Psi(\frac{q}{Q})\frac{\varphi(q^{*})}{\varphi(q)}
\end{equation}
in the range $(\log{Q})^{6} \leq T \leq (\log{Q})^{A}$, as $Q \rightarrow \infty$.  Here the constant
$c(\theta,r,R)$ is given by
\begin{equation}\label{118}
r^{2}c(\theta,r,R) = C(\theta,r,R) + e^{2R}C^{*}(\theta,r,R)
\end{equation}
with
\begin{equation}\label{119}
C(\theta,r,R) = -\Big(\frac{r^{2}}{2} + \frac{1}{4R^{2}}\Big)\Big(\frac{1}{\theta R} + \frac{\theta R}{3}\Big)
+ \frac{r^{2}}{2} - \frac{r}{2R}\Big(\frac{1}{\theta R} - \frac{\theta R}{3}\Big)
\end{equation}
and $C^{*}(\theta,r,R)$ is obtained from $C(\theta,r,R)$ by changing $r,R$ to $1-r, -R$ respectively; that is
\begin{equation}\label{120}
C^{*}(\theta,r,R) = \Big(\frac{(r-1)^{2}}{2} + \frac{1}{4R^{2}}\Big)\Big(\frac{1}{\theta R} + \frac{\theta R}{3}\Big)
+ \frac{(r-1)^{2}}{2} + \frac{r-1}{2R}\Big(\frac{1}{\theta R} - \frac{\theta R}{3}\Big).
\end{equation}
\end{theorem1}
In particular for $\theta = 1$ we have
\begin{equation}\label{121}
C(1,r,R) = -\Big(\frac{r^{2}}{2} + \frac{1}{4R^{2}}\Big)\Big(\frac{1}{R} + \frac{R}{3}\Big)
+ \frac{r^{2}}{2} - \frac{r}{2R}\Big(\frac{1}{R} - \frac{R}{3}\Big)
\end{equation}
\begin{equation}\label{122}
C^{*}(1,r,R) = \Big(\frac{(r-1)^{2}}{2} + \frac{1}{4R^{2}}\Big)\Big(\frac{1}{R} + \frac{R}{3}\Big)
+ \frac{(r-1)^{2}}{2} + \frac{r-1}{2R}\Big(\frac{1}{R} - \frac{R}{3}\Big).
\end{equation}

For $r =1$ the formula \eqref{118} simplifies a lot
\begin{equation}\label{123}
c(\theta,1,R) = \frac{e^{2R}}{4R^{2}}\Big(\frac{1}{\theta R} + \frac{\theta R}{3}\Big) - \frac{1}{4\theta R^{3}}
- \frac{1}{2\theta R^{2}} - \Big(\frac{1}{2\theta} + \frac{\theta}{12}\Big)\frac{1}{R} + \frac{\theta + 3}{6}
- \frac{\theta R}{6}.
\end{equation}
The original choice of Levinson was $\theta = 1/2$, $r=1$, in which case \eqref{123} becomes
$$c(\tfrac{1}{2},1,R) = \frac{e^{2R}}{4R^{2}}\Big(\frac{2}{R} + \frac{R}{6}\Big) - \frac{1}{2R^{3}}
- \frac{1}{R^{2}} - \frac{25}{24}\frac{1}{R} + \frac{7}{12} - \frac{R}{12}.$$

Recall that $I_{\chi}$ denotes the weighted mean-value of $|G(\sigma + it,\chi)M(\frac{1}{2}+it,\chi)|^{2}$
and $\Phihat(1)$ denotes the mean-value of $\Phi(t)$, so the asymptotic formula \eqref{117} asserts that
$c(\theta,r,R)$ is the mean-value of $I_{\chi}$.

Now it is quick to derive Theorem \ref{theorem1} from Theorem \ref{theorem2}.  By Corollary \ref{corollaryA}
in the last section we get
\begin{equation}\label{124}
\mathcal{N}_{0}'(T,Q) \geq (\kappa' + o(1))\mathcal{N}(T,Q)
\end{equation}
for $(\log{Q})^{6} \leq T \leq (\log{Q})^{A}$, $Q \rightarrow \infty$, where
\begin{equation}\label{125}
\kappa' = 1 - \frac{1}{R}\log{c(\theta,r,R)}.
\end{equation}
For $\theta = 1$, $r=10/9$ and $R=0.83$ this yields $c(\theta,r,R) = 1.44079\ldots$ and
$\kappa' = 0.56001\ldots$.

\begin{remark1}
Our choice of the mollifier \eqref{116} is relatively simple, but definitely not optimal.
The optimization analysis is given in the original paper \cite{C2}, see also \cite{C1}, and for the mollifier of the
form \eqref{19} the optimal choice turns out to be
$$P(x) = \frac{\sinh{ax}}{\sinh{a}} = \frac{e^{a}m^{\alpha}-e^{-a}m^{-\alpha}}{e^{a}-e^{-a}}$$
for $x = 1 - \log{m}/\log{X}$ with $\alpha = a/\log{X}$.  Numerically the best values are $a = 1.3408$,
$r^{-1} = 0.94$ and $R = 0.75$ giving \eqref{124} with
\begin{equation}\label{126}
\kappa' = 0.5865.
\end{equation}
\end{remark1}

\begin{remark1}
Our method also applies to twists of $GL_2$ and $GL_3$ L-functions. These cases are easier  because
the off-diagonal analysis is not necessary. However, the length of the mollifier is effectively shorter in these cases.
The principle is that a $GL_n$ mollifier of length $Q^{\theta-\epsilon}$ corresponds to a $GL_1$ mollifier of length $Q^{(\theta-\epsilon)/n}$.
By Theorem 2.3 of [CIS] we can take a $GL_2$ mollifier of length $Q^{1-\epsilon}$ and a $GL_3$ mollifier of length $Q^{1/2-\epsilon}$.
These correspond to  $GL_1$ situations with $\theta=1/2 $ and $\theta=1/6 $, respectively. Using the formula
$$\kappa'(\theta,r,R)=1-\frac{1}{R} c(\theta, r , R),$$
the $GL_1$ case may be written as
$$ \kappa'(1,1.06, 0.75) =0.5865\dots$$
We also have
$$ \kappa'(1/2,0.96, 1.24) =0.356\dots \qquad \kappa'(1/6,0.91, 2.37) =0.005\dots$$
Thus, using an obvious notation,
\begin{equation}
\mathcal{N}_f(T,Q) = \sum_{q}\frac{\Psi(q/Q)}{\varphi(q)}
\sumstar_{\chi (\bmod{q})}N_f(T,\chi)
\end{equation}
for counting the zeros of the twist $L_f(s,\chi)$ of an automorphic L-function $L_f(s)$, we have
\begin{theorem1}\label{theorem3} If $L_f(s)$ is a $GL_2$ L-function, then
for $Q$ and $T$ with $(\log{Q})^{6} \leq T \leq (\log{Q})^{A}$ we have
\begin{equation}
\mathcal{N}_{f,0}'(T,Q) \geq \frac{7}{20}\mathcal{N}_f(T,Q),
\end{equation}
 and if $L_f(s)$ is a $GL_3$ L-function, then
\begin{equation}
\mathcal{N}_{f,0}'(T,Q) \geq \frac{1}{200}\mathcal{N}_f(T,Q),
\end{equation}
\end{theorem1}
In other words, on average at least $35\%$ of the zeros of twists of a $GL_2$ L-function are simple and on the critical line, and at least one-half of one percent of the zeros of the twists of a given $GL_3$ L-function are simple and on the critical line.
\end{remark1}

Theorem \ref{theorem2} is the main ingredient in the proof of Theorem \ref{theorem1}.  In this paper we derive
Theorem \ref{theorem2} from more general results which we established in a separate paper \cite{CIS}.  These
results may have other applications and some parts of \cite{CIS} are better presented in a broader context.

We should say that the asymptotic formula \eqref{117} emerges from certain diagonal terms alone which are
relatively easy to compute, while the estimation of the off-diagonal terms constitutes the core of the matter.
The diagonal terms in question are the same for every $I_{\chi}$, so the averaging over $\chi \pmod{q}$ and
over $q \asymp Q$ does not play any role in estimating the percentage of the critical zeros.  These averagings are
needed solely to show that the contribution of the off-diagonal terms is negligible which is due to very strong
orthogonality of the characters and the randomness of the sign change of the M\"{o}bius function $\mu(m)$ in the mollifier.
To see the use of this feature we refer the reader to Section 8 of \cite{CIS}.

\vspace{.1in}

ACKNOWLEDGEMENTS. These works were begun at AIM in 1998 and continued over the years at AIM, Rutgers, IAS, Stanford, Bristol, and MSRI. We gratefully acknowledge the support of all of these institutions. This work was also supported in part by grants from the National Science Foundation.

\section{transformation of $I_{\chi}$}\label{section2}

We shall capture the derivative of $L(s,\chi)$ in $G(s,\chi)$ by Cauchy's formula
$$L'(s,\chi) = \frac{1}{2\pi i}\int_{|z|=\eps}L(s+z,\chi)z^{-2}dz.$$
To this end we need to consider a modified integral
\begin{equation}\label{21}
I_{\chi}(a,b) = \int L(a+it,\chi)L(1-b-it,\overline{\chi})|M(\tfrac{1}{2}+it,\chi)|^{2}\Phi(t)dt
\end{equation}
for complex numbers $a,b$ which are near $\sigma$ and $1-\sigma$ respectively.
We have
\begin{equation*}\begin{split}
|G(\sigma+ it,\chi)|^{2} & =  L(\sigma + it, \chi)L(\sigma - it, \chi)\\
& +  \lambda L(\sigma + it, \chi)L'(\sigma - it, \overline{\chi}) + \lambda L'(\sigma + it, \chi)L(\sigma - it, \chi)\\
& +  \lambda^{2} L'(\sigma + it, \chi)L'(\sigma - it, \overline{\chi}).
\end{split}\end{equation*}
Hence the integral \eqref{110} can be expressed in terms of \eqref{21} as follows
\begin{equation}\begin{split}\label{22}
I_{\chi}  =  I_{\chi}(\sigma,1-\sigma) &+ \lambda \oint I_{\chi}(\sigma, 1-\sigma + \beta)\beta^{-2}d\beta
+ \lambda \oint I_{\chi}(\sigma + \alpha, 1-\sigma)\alpha^{-2}d\alpha \\
&+ \lambda^{2}\oint\oint I_{\chi}(\sigma + \alpha, 1-\sigma + \beta)\alpha^{-2}\beta^{-2}d\alpha d\beta.
\end{split}\end{equation}

\section{Splitting $I_{\chi}(a,b)$}\label{section3}

We begin evaluation of $I_{\chi}(a,b)$ by opening the mollifier
\begin{equation}\label{31}
|M(\tfrac{1}{2}+it,\chi)|^{2} = \sum\sum_{h,k \leq X}c(h)c(k)\chi(h)\overline{\chi}(k)\Big(\frac{k}{h}\Big)^{it}
\end{equation}
where we put
\begin{equation}\label{32}
c(h) = \frac{\mu(h)}{\sqrt{h}}P(1 - \frac{\log{h}}{\log{X}}).
\end{equation}
Accordingly \eqref{21} splits into
\begin{equation}\label{33}
I_{\chi}(a,b) = \sum\sum_{h,k \leq X}c(h)c(k)I_{\chi}(a,b;h,k)
\end{equation}
where
\begin{equation}\label{34}
I_{\chi}(a,b;h,k) = \chi(h)\overline{\chi}(k)\int L(a + it,\chi)L(1-b-it,\overline{\chi})\Big(\frac{k}{h}\Big)^{it}\Phi(t)dt.
\end{equation}
Observe that $I_{\chi}(a,b;h,k)$ in $h,k$ depends only on the ratio $h/k$ in its lowest terms
\begin{equation}\label{35}
\frac{h}{k} = \frac{h_{1}}{k_{1}} \qquad \text{with} \,\, (h_{1},k_{1}) = 1
\end{equation}
provided we keep the redundant condition
\begin{equation}\label{36}
(hk,q) = 1.
\end{equation}

\section{Applying the functional equation}\label{section4}

For a primitive character $\chi \pmod{q}$ the $L$-function satisfies the following functional equation
\begin{equation}\label{41}
\Lambda(s,\chi) = \eps_{\chi}\Lambda(1-s,\chibar)
\end{equation}
where
\begin{equation}\label{42}
\Lambda(s,\chi) = \textbf{q}^{s/2}\Gamma(\frac{s+\nu}{2})L(s,\chi)
\end{equation}
with $\textbf{q} = q/\pi$ and $\nu = 0,1$ according to $\chi(-1) = 1,-1$.  Moreover $\eps_{\chi}$ is a complex number
with $|\eps_{\chi}| = 1$ (the sign of the Gauss sum).  Hence the product
\begin{equation}\label{43}
D_{A,B}(v) = \Lambda(A + v,\chi)\Lambda(1-B+v,\chibar)
\end{equation}
satisfies the functional equation
\begin{equation}\label{44}
D_{A,B}(v) = D_{B,A}(-v).
\end{equation}
We shall use this for $A = a+it$, $B=b+it$.  In this case
\begin{eqnarray}\nonumber
D_{A,B}(0) & = & \Lambda(A,\chi)\Lambda(1-B,\chibar)\\
\label{45} & = & \textbf{q}^{\frac{1}{2}(1+a-b)}\gamma_{ab}(t,\nu)L(a+it,\chi)L(1-b-it,\chibar)
\end{eqnarray}
where
\begin{equation}\label{46}
\gamma_{ab}(t,v) = \Gamma(\frac{a+it+v}{2})\Gamma(\frac{1-b-it+v}{2}).
\end{equation}
On the other hand we compute $D_{A,B}(0)$ by contour integration
\begin{equation}\label{47}
D_{A,B}(0) = \frac{1}{2\pi i}\int_{(\sigma)}[D_{A,B}(v) + D_{B,A}(v)]\frac{\omega(v)}{v}dv, \qquad \sigma >1.
\end{equation}
Here, for technical convenience we introduced a polar annihilator
\begin{equation}\label{48}
\omega(v) = \Big(1 - \Big(\frac{2v}{A-B}\Big)^{2}\Big)e^{v^{2}}.
\end{equation}
Note that $A - B = a-b \neq 0$, $\omega(v)$ is entire function of exponential decay in vertical strips, $\omega(v) = \omega(-v)$,
$\omega(0) = 1$, $\omega(\frac{a-b}{2})=0$.  For the proof of \eqref{47} start from the integral
$$\frac{1}{2\pi i}\int_{(\sigma)}D_{A,B}(v)\frac{\omega(v)}{v}dv,$$
move to the line $-\sigma$ passing a simple pole at $v=0$ with residue $D_{A,B}(0)$, then use the functional equation
\eqref{44} to return from the $-\sigma$ to the $\sigma$ line, getting \eqref{47}.

On the line $\Re{v} = \sigma > 1$ we expand \eqref{43} into Dirichlet series
$$D_{A,B}(v) = \textbf{q}^{\frac{1}{2}(1+a-b)}\gamma_{ab}(t, \nu + v)\sum_{m}\sum_{n}
\frac{\chi(m)}{m^{a}}\frac{\chibar(n)}{n^{1-b}}\Big(\frac{n}{m}\Big)^{it}\Big(\frac{\textbf{q}}{mn}\Big)^{v}.$$
Hence
\begin{eqnarray*}
\frac{1}{2\pi i}\int_{(\sigma)}D_{A,B}(v)\frac{\omega(v)}{v}dv & = &
\textbf{q}^{\frac{1}{2}(1+a-b)}\sum_{m}\sum_{n}\frac{\chi(m)}{m^{a}}\frac{\chibar(n)}{n^{1-b}}\Big(\frac{n}{m}\Big)^{it}\\
&& \frac{1}{2\pi i}\int_{(\sigma)}\gamma_{ab}(t,\nu + v)\Big(\frac{\textbf{q}}{mn}\Big)^{v}\frac{\omega(v)}{v}dv.\end{eqnarray*}
Add to this the same expression with $a,b$ interchanged to get $D_{A,B}(0)$ by \eqref{47}.  Inserting the result to \eqref{45}
we get
\begin{align*}
L(a+it,\chi&)L(1-b-it,\chibar) \\
&= \sum_{m}\sum_{n}\frac{\chi(m)}{m^{a}}\frac{\chibar(n)}{n^{1-b}}\Big(\frac{n}{m}\Big)^{it}
\frac{1}{2\pi i}\int_{(\sigma)}\frac{\gamma_{ab}(t,\nu + v)}{\gamma_{ab}(t,\nu + 0)}
\Big(\frac{\textbf{q}}{mn}\Big)^{v}\frac{\omega(v)}{v}dv\\
&+ \textbf{q}^{b-a}\sum_{m}\sum_{n}\frac{\chi(m)}{m^{b}}\frac{\chibar(n)}{n^{1-a}}\Big(\frac{n}{m}\Big)^{it}
\frac{1}{2\pi i}\int_{(\sigma)}\frac{\gamma_{ba}(t,\nu + v)}{\gamma_{ab}(t,\nu + 0)}
\Big(\frac{\textbf{q}}{mn}\Big)^{v}\frac{\omega(v)}{v}dv.\end{align*}
Note that the above two lines are not symmetric in $a,b$.  Inserting these lines to \eqref{34} we obtain
\begin{eqnarray}\label{49}
I_{\chi}(a,b;h,k) &=& \sum_{m}\sum_{n}\frac{\chi(hm)}{m^{a}}\frac{\chibar(kn)}{n^{1-b}}H_{ab}(\frac{kn}{hm},\frac{mn}{\textbf{q}})\\
\nonumber &+&
 \textbf{q}^{b-a}\sum_{m}\sum_{n}\frac{\chi(hm)}{m^{b}}\frac{\chibar(kn)}{n^{1-a}}H^{*}_{ab}(\frac{kn}{hm},\frac{mn}{\textbf{q}})
\end{eqnarray}
where $H_{ab}(x,y)$ is the function defined by
\begin{equation}\label{410}
H_{ab}(x,y) = \frac{1}{2\pi i}\int_{(\sigma)}\Big(\int\Phi(t)x^{it}\frac{\gamma_{ab}(t,\nu + v)}{\gamma_{ab}(t,\nu + 0)}dt\Big)
y^{-v}\frac{\omega(v)}{v}dv.
\end{equation}
The adjoint function $H^{*}_{ab}(x,y)$ is defined similarly by interchanging $a,b$ in the numerator of \eqref{410} but
keeping the same denominator.

For notational simplicity and direct reference we set
\begin{equation}\label{411}
h_{ab}(t,v) = \frac{\gamma_{ab}(t,\nu + v)}{\gamma_{ab}(t,\nu + 0)} =
\frac{\Gamma(\frac{1}{2}(a+it+\nu+v))\Gamma(\frac{1}{2}(1-b-it+\nu+v))}
{\Gamma(\frac{1}{2}(a+it+\nu))\Gamma(\frac{1}{2}(1-b-it+\nu))}
\end{equation}
and
\begin{equation}\label{412}
h^{*}_{ab}(t,v) = \frac{\gamma_{ba}(t,\nu + v)}{\gamma_{ab}(t,\nu + 0)} =
\frac{\Gamma(\frac{1}{2}(b+it+\nu+v))\Gamma(\frac{1}{2}(1-a-it+\nu+v))}
{\Gamma(\frac{1}{2}(a+it+\nu))\Gamma(\frac{1}{2}(1-b-it+\nu))}
\end{equation}
Hence \eqref{410} becomes
\begin{equation}\label{413}
H_{ab}(x,y) =  \frac{1}{2\pi i}\int_{(\sigma)}\Big(\int\Phi(t)x^{it}h_{ab}(t,v)dt\Big)y^{-v}\frac{\omega(v)}{v}dv.
\end{equation}
Similarly $H^{*}_{ab}(x,y)$ is given by \eqref{413} with $h_{ab}(t,v)$ replaced by $h^{*}_{ab}(t,v)$.

Note that for $v=0$ we have $h_{ab}(t) = h_{ab}(t,0) = 1$ and
\begin{equation}\label{414}
h^{*}_{ab}(t) = h^{*}_{ab}(t,0) =
\frac{\Gamma(\frac{1}{2}(b+it+\nu))\Gamma(\frac{1}{2}(1-a-it+\nu))}
{\Gamma(\frac{1}{2}(a+it+\nu))\Gamma(\frac{1}{2}(1-b-it+\nu))}.
\end{equation}
By Stirling's formula this yields
\begin{equation}\label{415}
h^{*}_{ab}(t) = |t|^{b-a} + O((|t|+1)^{\Re{(b-a-1)}}).
\end{equation}

\section{Estimates for $H_{ab}(x,y)$}\label{section5}

We shall need the following estimates for partial derivatives of $H_{ab}(x,y)$ and $H^{*}_{ab}(x,y)$.
\begin{lemma}\label{lemma51}
Suppose $\Phi(t)$ satisfies \eqref{112} with $T \geq 3$.  We have
\begin{equation}\label{51}
x^{i}y^{j}\frac{\partial^{i+j}}{\partial x^{i} \partial y^{j}}H_{ab}(x,y) \ll (1 + |\log{x}|)^{-A}(1 + \frac{y}{T})^{-B}T^{1+i}
\end{equation}
for any $i,j,A,B \geq 0$ and the implied constant depending on $i,j,A,B$.  Moreover the same estimates hold for $H^{*}_{ab}(x,y)$.
\end{lemma}

\begin{proof}
We give details for $H_{ab}(x,y)$, the case of $H^{*}_{ab}(x,y)$ is similar.  The left-hand side of \eqref{51} is equal to
\begin{equation}\label{52}
\frac{1}{2\pi i}\int_{(\sigma)}\Big(\int\Phi(t)x^{it}h_{ab}(t,v)p(t)dt\Big)y^{-v}r(v)\frac{\omega(v)}{v}dv
\end{equation}
where $p(t), r(v)$ are polynomials of degree $i$ and $j$, respectively.  Specifically $p(t) = P_{i}(it)$ and
$r(v) = P_{j}(-v)$, where
$$P_{n}(x) = \prod_{0 \leq d \leq n}(x-d).$$
Then we integrate by parts in the $t$-variable $k$ times, getting
\begin{equation}\label{53}
\int dt = \Big(\frac{i}{\log{x}}\Big)^{k}\int\big(\Phi(t)h_{ab}(t,v)p(t)\big)^{(k)}x^{it}dt.
\end{equation}
For any $l \geq 0$ we have
\begin{equation}\label{54}
\big(h_{ab}(t,v)\big)^{(l)} \ll (|t|+1)^{\sigma - l}(|v|+1)^{\sigma + l}
\end{equation}
where $\sigma = \Re{v} \geq -1/4$.  For $l=0$ this follows by Stirling's formula
$$\Gamma(s) = \Big(\frac{2\pi}{s}\Big)^{1/2}\Big(\frac{s}{e}\Big)^{s}\Big(1 + O\Big(\frac{1}{|s|}\Big)\Big), \qquad
|\arg{s}| \leq \pi - \eps.$$
We are going to verify \eqref{54} for $l=1$, the case of higher derivatives is similar.  To this end it is enough to estimate
the logarithmic derivative
\begin{equation}\begin{split}\label{55}
\Big(\log{h_{ab}(t,v)}\Big)' &=  \frac{i}{2}\psi(\frac{a+it+\nu+v}{2}) - \frac{i}{2}\psi(\frac{1-b-it+\nu+v}{2})\\
&-\frac{i}{2}\psi(\frac{a+it+\nu}{2}) + \frac{i}{2}\psi(\frac{1-b-it+\nu}{2})
\end{split}\end{equation}
where
$$\psi(s) = \frac{\Gamma'}{\Gamma}(s) = \log{s} + O\Big(\frac{1}{|s|}\Big) = \log{|s|} +i\arg{s} + O\Big(\frac{1}{|s|}\Big).$$
Hence it is easy to see that \eqref{55} is bounded by $(|v|+1)/(|t|+1)$.  Multiplying \eqref{54} for $l=0$ by this bound
we get \eqref{54} for $l=1$.

By \eqref{112}, \eqref{54} and \eqref{53} we derive
$$\int dt \ll \frac{(|v|+1)^{\sigma + k}}{|\log{x}|^{k}}\int_{0}^{\infty}\Big(1 + \frac{t}{T}\Big)^{-C}(t+1)^{\sigma - k+i}dt
\ll \frac{(|v|+1)^{\sigma+k}}{|\log{x}|^{k}}T^{\sigma + 1+i}.$$
Inserting this bound to \eqref{52} and estimating trivially in $v$ (recall that $\omega(v)$ has more than exponential decay as
$|v| \rightarrow \infty$ in vertical strips) we find that the left-hand side of \eqref{51} is bounded by
$$y^{-\sigma}T^{\sigma+1+i}|\log{x}|^{-k}.$$
This gives the right-hand side of \eqref{51} by the following choices of $k$ and $\sigma$;
$$\begin{array}{lrl}
k = 0  \,\,\text{if} & |\log{x}| \leq 1,& \quad k=A \,\,\text{otherwise}\\
\sigma = 0 \,\,\text{if} & y \leq T,& \quad \sigma = B \,\,\text{otherwise}.
\end{array}$$
\end{proof}

\begin{remark1}
Formally speaking our choice $\sigma = 0$ if $y \leq T$ and $j=0$ is not allowed because of the simple pole at $v=0$.
To be precise in this case, move to the line $\sigma = -1/4$ getting a better bound than claimed, except for the contribution
of the residue at $v=0$ which gives the bound as claimed.
\end{remark1}

\begin{lemma1}\label{lemma51star}
The bound \eqref{51} holds for $H^{*}_{ab}(x,y)$, but with the extra factor $T^{\Re{(b-a)}}$.
\end{lemma1}

\begin{proof}
Use the same arguments as for $H_{ab}(x,y)$ and along the lines apply \eqref{415}.
\end{proof}

\begin{remark1}
The excess factor $T^{i}$ in \eqref{51} is not going to cause a problem because in applications $T$ is relatively
small, $T \ll (\log{Q})^{C}$.  Moreover if $|\log{x}| \gg Q$ then a loss of any power of $T$ is compensated by the gain in
powers of $|\log{x}|$.
\end{remark1}

If $y > T$ the above arguments yield
\begin{equation}\label{56}
H_{ab}(x,y) \ll (1 + T|\log{x}|)^{-A}\Big(\frac{T}{y}\Big)^{B}T
\end{equation}
for any $A,B \geq 0$.  The same bound holds for $H^{*}_{ab}(x,y)$.

\begin{remark1}
For $i=j=0$ the bound \eqref{51} becomes
\begin{equation}\label{57}
H_{ab}(x,y) \ll (1 + |\log{x}|)^{-A}\Big(1 + \frac{y}{T}\Big)^{-B}T.
\end{equation}
The same bound holds for $H^{*}_{ab}(x,y)$.  These bounds show that the series \eqref{49} runs effectively over $m,n$ in the range
\begin{equation}\label{58}
mn \leq (QT)^{1+\eps}
\end{equation}
The contribution of the tail of the series is negligible.  Moreover, after applications of the Asymptotic Large Sieve
(which is developed in \cite{CIS}), one only needs the range $|\log{x}| \leq \eps\log{QT}$, i.e.
\begin{equation}\label{59}
(QT)^{-\eps} \leq \frac{hn}{km} \leq (QT)^{\eps}.
\end{equation}
\end{remark1}

\section{Selecting the diagonal}\label{section6}

The main objective of \cite{CIS} is to evaluate character sums of general type and some special type like $T_{\chi}(a,b;h,k)$ given by
\eqref{49}.  Our test function $H_{ab}(x,y)$ in \eqref{49} satisfies the conditions described in Section 9 of \cite{CIS}.
Moreover the coefficients $c(h)$ in our mollifier (see \eqref{32}) also satisfies the conditions (2.24)-(2.26) of
\cite{CIS} (apart of the normalization).  Therefore according to Theorem 2.5 of \cite{CIS} the main contribution to
\eqref{49} comes from the diagonal terms $hm=kn$ giving
\begin{equation}\label{61}\begin{split}
I^{=}(a,b;h,k) = &\sum\sum_{\substack{hm=kn\\(mn,q)=1}}m^{-a}n^{b-1}H_{ab}(1,\frac{mn}{\textbf{q}})\\
&+\textbf{q}^{b-a}\sum\sum_{\substack{hm=kn\\(mn,q)=1}}m^{-b}n^{a-1}H^{*}_{ab}(1,\frac{mn}{\textbf{q}}).
\end{split}\end{equation}
The off-diagonal terms $hm \neq kn$ in \eqref{49} are not really small for a given character $\chi \pmod{q}$, but they cancel
out considerably in average over $\chi \pmod{q}$, $q \asymp Q$ and $h,k$ with
\begin{equation}\label{62}
h,k \leq X \leq Q^{1-\eps}.
\end{equation}
Denote this average by
\begin{eqnarray}\label{63}
\mathcal{L}_{ab}(Q,T) &= &\sum_{q}\frac{\Psi(q/Q)}{\varphi(q)}\sumstar_{\chi (\bmod{q})}I_{\chi}(a,b)\\\nonumber
& = & \sum_{q}\frac{\Psi(q/Q)}{\varphi(q)}\sumstar_{\chi (\bmod{q})}\sum\sum_{h,k \leq X}c(h)c(k)I_{\chi}(a,b;h,k).
\end{eqnarray}
The corresponding diagonal term is
\begin{equation}\label{64}
\mathcal{L}^{=}_{ab}(Q,T) = \sum_{q}\Psi(q/Q)\frac{\varphi^{*}(q)}{\varphi(q)}
\sum\sum_{\substack{h,k \leq X\\(hk,q)=1}}c(h)c(k)I^{=}(a,b;h,k).
\end{equation}
By Theorem 2.5 of \cite{CIS} we get
\begin{equation}\label{65}
\mathcal{L}_{ab}(Q,T) = \mathcal{L}^{=}_{ab}(Q,T) + (QT(\log{Q})^{-C})
\end{equation}
where $C$ is any positive constant.

It remains to evaluate $\mathcal{L}^{=}_{ab}(Q,T)$.  In this task we no longer need any help from averaging over the
conductor $q \asymp Q$.  We shall handle separately every sum
\begin{equation}\label{66}
E_{q}(X) = \sum\sum_{\substack{h,k \leq X\\(hk,q)=1}}c(h)c(k)I^{=}(a,b;h,k).
\end{equation}
Note that by trivial estimation using \eqref{57} we get
\begin{equation}\label{67}
I^{=}(a,b;h,k) \ll \frac{(h,k)}{\sqrt{hk}}T\log{qT}.
\end{equation}
Hence
\begin{equation}\label{68}
E_{q}(X) \ll T(\log{qT})(\log{X})^{3},
\end{equation}
while our goal is to show that
\begin{equation}\label{69}
E_{q}(X) \sim c(\theta,r,R)\Phihat(1)
\end{equation}
which would finish the proof of \eqref{117}.  Therefore we need to save slightly more than $(\log{Q})^{6}$ by comparison
of \eqref{68} and \eqref{69}.

\section{Computing $I^{=}(a,b;h,k)$}\label{section7}

Recall that (see \eqref{413} for $x=1$)
\begin{equation}\label{71}
H_{ab}(1,y) = \frac{1}{2\pi i}\int_{(\sigma)}\Big(\Phi(t)h_{ab}(t,v)dt\Big)y^{-v}\frac{\omega(v)}{v}dv
\end{equation}
and the corresponding integral formula holds for $H^{*}_{ab}(1,y)$.  Inserting these integral representations to \eqref{61}
we get
\begin{equation}\label{72}\begin{split}
I^{=}(a,b;h,k) &= \int\Phi(t)\Big(\frac{1}{2\pi i}\int_{(\sigma)}h_{ab}(t,v)Z_{ab}(v)\textbf{q}^{v}\frac{\omega(v)}{v}dv\Big)dt\\
& +\textbf{q}^{b-a}\int\Phi(t)\Big(\frac{1}{2\pi i}\int_{(\sigma)}h^{*}_{ab}(t,v)Z_{ba}(v)\textbf{q}^{v}\frac{\omega(v)}{v}dv\Big)dt
\end{split}\end{equation}
where
\begin{eqnarray*}
Z_{ab}(v) &=& \sum\sum_{\substack{hm=kn\\(mn,q)=1}}m^{-a}n^{b-1}(mn)^{-v}\\
& = & k_{1}^{-a}h_{1}^{b-1}(h_{1}k_{1})^{-v}\zeta_{q}(2v+1+a-b)
\end{eqnarray*}
where $\zeta_{q}(s)$ stands for the zeta function with the local factors at primes dividing $q$ being omitted.  Recall also
the notation \eqref{35}, that is $h_{1} = h/(h,k)$, $k_{1}=k/(h,k)$, and we keep the condition $(hk,q)=1$.  For notational
convenience we put
\begin{equation}\label{73}
l = \sqrt{h_{1}k_{1}}.
\end{equation}

Next we compute the contour integrals by moving to the line $\sigma = -1/4$ passing simple poles at $v=0$ with residues
\begin{equation}\label{74}
h_{ab}(t,0)k_{1}^{-a}h_{1}^{b-1}\zeta_{q}(1+a-b)
\end{equation}
\begin{equation}\label{75}
h^{*}_{ab}(t,0)k_{1}^{-b}h_{1}^{a-1}\zeta_{q}(1+b-a).
\end{equation}
Note that the pole of $\zeta_{q}(2v+1+a-b)$ at $v=\frac{1}{2}(b-a)$ is annihilated by the zero of $\omega(v)$.

The integrals on the line $\sigma = -1/4$ are bounded by (use \eqref{54}) $O((h_{1}k_{1}/q)^{1/4})$ which is sufficient if
\begin{equation}\label{76}
h_{1}k_{1} \leq q^{1/4}.
\end{equation}
If \eqref{76} does not hold we stop at the line $\Re{v} = 1/\log{q}$.  The resulting trivial estimation is not satisfactory,
but only by a factor $(\log{q})^{c}$, where $c$ is an absolute constant.  However in the range $h_{1}k_{1} > q^{1/4}$
we can gain a factor $(\log{q})^{-C}$ with any large constant $C$ due to the cancellation in the sum of
$\mu(h_{1}k_{1})/h_{1}k_{1}$ which appears in the mollifier.  Having said that we are left with the polar terms
\begin{equation}\label{77}\begin{split}
I^{=}(a,b;h,k) &= k_{1}^{-a}h_{1}^{b-1}\zeta_{q}(1+a-b)\int\Phi(t)h_{ab}(t,0)dt\\
&+\textbf{q}^{b-a}k_{1}^{-b}h_{1}^{a-1}\zeta_{q}(1+b-a)\int\Phi(t)h^{*}_{ab}(t,0)dt\\
&+ \Delta_{q}(a,b;h,k)
\end{split}\end{equation}
where the remainder term $\Delta_{q}(a,b;h,k)$ is small after summation in $h$ and $k$;
\begin{equation}\label{78}
\sum\sum_{\substack{h,k\leq X\\(hk,q)=1}}c(h)c(k)\Delta_{q}(a,b;h,k) \ll (\log{q})^{-C}.
\end{equation}
Note that we no longer need the restriction \eqref{76} for the main terms in \eqref{77} because it can be relaxed for the
same reason which allowed us to introduce it.

At the end of Section \ref{section4} we have noticed that $h_{ab}(t,0)$ and $h^{*}_{ab}(t,0)$ satisfies \eqref{415}.  Hence
the first integral in \eqref{77} is equal to
\begin{equation}\label{79}
\Phihat(1) = \int\Phi(t)dt
\end{equation}
and the second integral is approximately equal to
\begin{equation}\label{710}
\Phihat(1+b-a) = \int\Phi(t)|t|^{b-a}dt
\end{equation}
up to an error term $O(\log{T})$.  This error term is smaller than the main term by factor $T^{-1}\log{T}$ which makes it
negligible if $T \geq (\log{Q})^{6}$.

Recall that $b-a \asymp (\log{Q})^{-1}$ while $|t| \leq T \leq (\log{Q})^{A}$, so $\Phihat(1+b-a)$ does also approximate
to $\Phihat(1)$, but not good enough to ignore the difference
\begin{equation}\label{711}
\phi(b-a) = \int\Phi(t)\big(|t|^{b-a}-1\big)dt,
\end{equation}
at least not yet at current state of our considerations.  Nevertheless we re-write \eqref{77} in the following form
\begin{equation}\label{712}\begin{split}
I^{=}(a,b;h,k) & = \Phihat(1)l^{-1}V_{q}(a,b;h,k)\\
& + \phi(b-a)q^{b-a}h_{1}^{-b}k_{1}^{a-1}\zeta_{q}(1+b-a)\\
& + \Delta_{q}(a,b;h,k) + O(l^{-1}(\log{Q})(\log{T})).
\end{split}\end{equation}
where
\begin{equation}\label{713}
V_{q}(a,b;h,k) = h_{1}^{b-1/2}k_{1}^{1/2-a}\zeta_{q}(1+a-b)+ \textbf{q}^{b-a}h_{1}^{a-1/2}k_{1}^{1/2-b}\zeta_{q}(1+b-a).
\end{equation}

The leading term as well as the second one in \eqref{712} can be handled in very similar ways, so we only go for the leading term
$V_{q}(a,b;h,k)$.  The second term does not contribute to the final main term, it yields less by factor $\log{Q}/\log T,$
due to $\phi(b-a) \ll |b-a| \log T$.

We are going to allow another technical shortcut concerning the co-primality restriction $(hk,q)=1$ and a similar one in
$\zeta_{q}(s)$.  These restrictions can be relaxed without affecting the final asymptotic formula \eqref{69}.
The point is that the action of the mollifier of the zeta function reduces substantially the weights attached to numbers
having small prime factors.  For this reason we are going to suppress the condition $(hk,q)=1$ in \eqref{66} and delete the
subscript $q$ in \eqref{713}.  A precise justification is left as an exercise.

\section{Computing derivatives of $V(a,b;h,k)$}\label{secion8}

We need to evaluate
\begin{equation}\label{81}
V(a,b;h,k) = h_{1}^{b-1/2}k_{1}^{1/2-a}\zeta(1+a-b)+ \textbf{q}^{b-a}h_{1}^{a-1/2}k_{1}^{1/2-b}\zeta_{q}(1+b-a)
\end{equation}
for $a = \sigma + \alpha$ and $b = 1-\sigma - \beta$.  In this case \eqref{81} becomes
\begin{equation}\label{82}
F(\alpha,\beta) = l^{1-2\sigma}h_{1}^{-\beta}k_{1}^{-\alpha}\zeta(2\sigma + \alpha+\beta)+\Big(\frac{l}{\textbf{q}}\Big)^{2\sigma -1}
\frac{h_{1}^{\alpha}k_{1}^{\beta}}{\textbf{q}^{\alpha+\beta}}\zeta(2 - 2\sigma -\alpha-\beta).
\end{equation}
Two operations need to be performed; summation over $h,k$ according to \eqref{78} and computing the derivatives in $\alpha,\beta$
according to \eqref{22}.  We have chosen to do the latter first because it yields an exact simple formula (well, only for a
convenient choice of the parameter $\lambda$).

According to \eqref{22} we need to compute the following linear combination
\begin{eqnarray}\label{83}
V(h,k)& =& F^{(00)}(0,0) + \lambda F^{(10)}(0,0) + \lambda F^{(01)}(0,0) + \lambda^{2} F^{(11)}(0,0)\\\nonumber
& = & \lambda^{2}\big(e^{\frac{\alpha+\beta}{\lambda}}F(\alpha,\beta)\big)^{(11)}, \qquad \text{at}\,\, \alpha = \beta =0.
\end{eqnarray}
We choose $\lambda = (\log{\textbf{q}^{r}})^{-1}$ (see \eqref{16} and \eqref{17}) getting
\begin{align}\label{84}
(\log{\textbf{q}^{r}})^{2}V(h,k) &= \big(\textbf{q}^{(\alpha + \beta)r}F(\alpha,\beta)\big)^{(11)}\\\nonumber
& =  \Big(l^{1-2\sigma}\textbf{q}^{(\alpha+\beta)r}h_{1}^{-\beta}k_{1}^{-\alpha}\zeta(2\sigma + \alpha+\beta)\\\nonumber &+
\Big(\frac{l}{\textbf{q}}\Big)^{2\sigma -1}\textbf{q}^{(\alpha+\beta)(r-1)}h_{1}^{\alpha}k_{1}^{\beta}
\zeta(2 - 2\sigma -\alpha-\beta)\Big)^{(11)}\\
\nonumber &=  l^{1-2\sigma}\Big[(\log{\frac{\textbf{q}^{r}}{h_{1}}})(\log{\frac{\textbf{q}^{r}}{k_{1}}})\zeta(2\sigma)
+\big(\log{\frac{\textbf{q}^{r}}{h_{1}}}+\log{\frac{\textbf{q}^{r}}{k_{1}}}\big)\zeta'(2\sigma) + \zeta''(2\sigma) \Big]\\
\nonumber &+ \textbf{q}^{1-2\sigma}\{\text{above line with $\sigma, r$ replaced by $1-\sigma, 1-r$ respectively}\}.
\end{align}
Recall that $h_{1} = h/(h,k)$, $k_{1}= k /(h,k)$ and $l = \sqrt{h_{1}k_{1}} = \sqrt{hk}/(h,k)$.

\section{Summing over the mollifier}\label{section9}

Next we need to evaluate the sum
\begin{equation}\label{91}
V = \sum_{h}\sum_{k}c(h)c(k)l^{-1}V(h,k),
\end{equation}
see \eqref{66}, \eqref{712}, \eqref{713}.  According to \eqref{84} this splits into
\begin{align}\label{92}
&(\log{\textbf{q}^{r}})^{2}V\\\nonumber
& = [V_{0}(\log{\textbf{q}^{r}})^{2} - 2V_{1}\log{\textbf{q}^{r}} + V_{2}]\zeta(2\sigma) +
2(V_{0}\log{\textbf{q}^{r}}-V_{1})\zeta'(2\sigma)
+V_{0}\zeta''(2\sigma)\\\nonumber
&+\textbf{q}^{1-2\sigma}\{\text{above line with $\sigma, r$ replaced by $1-\sigma, 1-r$ respectively}\}.
\end{align}
Here $V_{0},V_{1},V_{2}$ are the sums of type \eqref{91} with
$$V_{0}(h,k) = l^{1-2\sigma}, \quad V_{1}(h,k) = l^{1-2\sigma}\log{l},
\quad V_{2}(h,k) = l^{1-2\sigma}(\log{h_{1}})(\log{k_{1}}).$$
Have in mind that $V_{0},V_{1},V_{2}$ depend on $\sigma$, so they change in the last line of \eqref{92} by replacing $\sigma$
to $1 - \sigma$ (as do the values of derivatives of the zeta function).

To evaluate the corresponding sums $V_{0},V_{1},V_{2}$ we appeal to Lemma 1 of \cite{C1}.  We only need a special case of this
lemma in which the polynomials $P_{1}(x)$, $P_{2}(x)$ are both equal to $P(x) = x$, (see \eqref{116}).  In this case
Lemma 1 of \cite{C1} yields
\begin{eqnarray}\label{93}
\sum\sum_{h,k \leq X}c(h)c(k)l^{-1}h_{1}^{-\alpha}k_{1}^{-\beta} & \sim &
\frac{1}{\log{X}}\int_{0}^{1}(1 + \alpha x\log{X})(1 + \beta x\log{X})dx\\\nonumber
& = & \frac{1}{\log{X}} + \frac{\alpha + \beta}{2} + \frac{\alpha\beta}{3}\log{X}.
\end{eqnarray}
This formula holds uniformly in complex numbers $\alpha,\beta \ll (\log{X})^{-1}$.  Moreover \eqref{93} admits differentiations
in $\alpha,\beta$.
Choosing $\alpha = \beta = 1/2 - \sigma$ we get
\begin{equation}\label{94}
V_{0} \sim (\log{X})^{-1} - \Big(\frac{1}{2}-\sigma\Big) + \frac{1}{3}\Big(\frac{1}{2}-\sigma\Big)^{2}\log{X}.
\end{equation}
Differentiating $V_{0}$ with respect to $\sigma$ and dividing by $-2$ we get
\begin{equation}\label{95}
V_{1} \sim -\frac{1}{2} + \frac{1}{3}\Big(\frac{1}{2}-\sigma\Big)\log{X}.
\end{equation}
Differentiating \eqref{93} in $\alpha,\beta$ and choosing $\alpha = \beta = 1/2-\sigma$ we get
\begin{equation}\label{96}
V_{2} \sim \frac{1}{3}\log{X}.
\end{equation}

Now we are ready to compute the sum \eqref{91} from the partition \eqref{92} using the asymptotic values given above.
We also approximate $\zeta(s)$ by its polar term $(s-1)^{-1}$ getting the following asymptotic values
$$\zeta(2\sigma) \sim (2\sigma-1)^{-1}, \quad \zeta'(2\sigma) \sim -(2\sigma-1)^{-2}, \quad \zeta''(2\sigma)
\sim 2(2\sigma-1)^{-3}.$$
We choose
\begin{equation}\label{97}
X = \textbf{q}^{\theta} \qquad \text{with}\,\, \theta >0
\end{equation}
and
\begin{equation}\label{98}
\sigma = \frac{1}{2} - \frac{R}{\log{\textbf{q}}} \qquad \text{with}\,\, R >0.
\end{equation}
Then \eqref{94},\eqref{95},\eqref{96} become
\begin{eqnarray*}
V_{0} & \sim & \Big(\frac{1}{\theta R}-1+\frac{\theta R}{3}\Big)\frac{R}{\log{\textbf{q}}},\\
V_{1} & \sim & -\Big(\frac{1}{2}-\frac{\theta R}{3}\Big),\\
V_{2} & \sim & \frac{\theta}{3}\log{\textbf{q}}.
\end{eqnarray*}
Moreover we get
$$\zeta(2\sigma) \sim -\frac{\log{\textbf{q}}}{2R}, \quad \zeta'(2\sigma) \sim -\Big(\frac{\log{\textbf{q}}}{2R}\Big)^{-2},
\quad \zeta''(2\sigma) \sim 2\Big(\frac{\log{\textbf{q}}}{2R}\Big)^{-3}.$$
Note that the corresponding asymptotic values when $\sigma$ is changed to $1-\sigma$ are obtained by changing $R$ to $-R$.

Introducing the above asymptotic values to \eqref{92} we find that $V \sim c(\theta,r,R)$, where $c(\theta,r,R)$ is computed
as follows
$$r^{2}c(\theta,r,R) = C(\theta,r,R) + e^{2R}C(\theta,1-r,-R)$$
with
\begin{eqnarray*}
C(\theta,r,R) & = &
-\frac{1}{2R}\Big[r^{2}\Big(\frac{1}{\theta R} -1+\frac{\theta R}{3}\Big)R + 2r\Big(\frac{1}{2}-\frac{\theta R}{3}\Big)
+ \frac{\theta}{3}\Big]\\
&&-\frac{1}{2R^{2}}\Big[r\Big(\frac{1}{\theta R} -1+\frac{\theta R}{3}\Big)R + \Big(\frac{1}{2}-\frac{\theta R}{3}\Big)
\Big]\\
&&-\frac{1}{4R^{3}}\Big(\frac{1}{\theta R} -1+\frac{\theta R}{3}\Big)R\\
&=&-\frac{1}{4R^{2}}\Big(\frac{1}{\theta R} -1+\frac{\theta R}{3}\Big)(2r^{2}R^{2}+2rR+1)\\
&&-\frac{1}{2R^{2}}\Big(\frac{1}{2}-\frac{\theta R}{3}\Big)(2rR+1)-\frac{\theta}{6R}\\
&=&-\frac{r^{2}}{2}\Big(\frac{1}{\theta R} -1+\frac{\theta R}{3}\Big) - \frac{2rR+1}{4R^{2}}\Big(\frac{1}{\theta R} -
\frac{\theta R}{3}\Big) - \frac{\theta}{6R}\\
&=&-\frac{r^{2}}{2}\Big(\frac{1}{\theta R}+\frac{\theta R}{3}\Big) +\frac{r^{2}}{2} - \frac{r}{2R}\Big(\frac{1}{\theta R} -
\frac{\theta R}{3}\Big) - \frac{1}{4R^{2}}\Big(\frac{1}{\theta R} +\frac{\theta R}{3}\Big)
\end{eqnarray*}
which agrees with \eqref{119}.  This completes the proof of \eqref{117} and of Theorem \ref{theorem2}.

\appendix
\section{Levinson's Method}

This is all about estimating the number of zeros in segments of the critical line for $L$-functions having Euler product
and satisfying suitable functional equations.  In this section we are going to sketch the basic ideas of Levinson's
method \cite{L}.

Let $L(s,f)$ be given by the Dirichlet series
\begin{equation}\label{A1}
L(s,f) = \sum_{1}^{\infty}\lambda_{f}(n)n^{-s}
\end{equation}
which converges absolutely in $\Re{s} > 1$ and it has the Euler product of degree $d$, that is
\begin{equation}\label{A2}
L(s,f) = \prod_{p}(1 - \alpha_{1}(p)p^{-s})^{-1}\cdots (1 - \alpha_{d}(p)p^{-s})^{-1}.
\end{equation}
Therefore the coefficients $\lambda_{f}(n)$ are multiplicative.  Moreover we assume that $L(s,f)$ has analytic continuation
to the whole complex $s$-plane with at most one pole at $s=1$ of order $\leq d$.  Next there is a corresponding local factor
at $p = \infty$, say $\gamma(s) = \gamma(s,f)$ which is given by the product of $d$ gamma functions of the following type
\begin{equation}\label{A3}
\gamma(s) = \pi^{-ds/2}\Gamma(\frac{s+\kappa_{1}}{2})\cdots\Gamma(\frac{s+\kappa_{d}}{2})
\end{equation}
with the parameters $\kappa_{j}$ having $\Re{\kappa_{j}}> -1/2$ and the non-real ones occur in complex conjugate pairs.
In addition to the above data there is a conductor $q = q(f)$ which is a positive integer and a root number
$\eps = \eps(f)$ which is a complex number with $|\eps| =1$.  We shall write $\eps = \overline{\eta}/\eta$ with $\eta \in \C^{*}$.
Having all the above factors we assume that the following functional equation holds
\begin{equation}\label{A4}
\eta X(s)L(s,f) = \overline{\eta}X(1-s)L(1-s,g)
\end{equation}
where
\begin{equation}\label{A5}
X(s) = q^{s/2}\gamma(s)
\end{equation}
and $L(s,g)$ is the $L$-function with coefficients $\lambda_{g}(n) = \overline{\lambda_{f}}(n)$.

The fundamental question is where are the zeros of $L(s,f)$?  Since $X(s)$ never vanishes the zeros $\rho_{f}$ of
$L(s,f)$ in the strip $0 \leq \Re{s} \leq 1$ correspond to the zeros $\rho_{g} = 1 -\rho_{f}$ of $L(s,g)$.  The
Riemann Hypothesis, if true, would say that $\rho_{g} = \overline{\rho_{f}}$.

Let $N(T,f)$ denote the number of all zeros $\rho = \beta + i\gamma$ of $L(s,f)$ with $0 \leq\beta \leq 1$, $|\gamma| \leq T$,
each one counted with the multiplicity equal to its order.  Let $N_{0}(T,f)$ denote the number of these zeros with $\beta =1/2$.
Following the memoir of B. Riemann \cite{R} one can easily derive a quite precise estimate (cf. \cite{IK})
\begin{equation}\label{A6}
N(T,f) = \frac{dT}{\pi}\log{\frac{T}{2\pi e}}+ \frac{T}{\pi}\log{q} + O(\log{qT})
\end{equation}
for all $T\geq 2$, the implied constant depending on the local parameters $\kappa_{1},\ldots,\kappa_{d}$.  It is
important to realize that the first part of \eqref{A6} comes from (approximately equal to) the variation of the
argument of $\gamma(s)$ over the vertical segment $s = -\eps+it$, $|t| \leq T$, while the second part is the variation
of the argument of $q^{-s/2}$.  Hence one knows that an overwhelming majority of zeros accounted by $N(T,f)$ are
captured by analytic behaviour of the single factor $X(s)$.  The variation of finite places in the Euler product
contribute very little to counting all the zeros.  However they do play a role in our counting the critical zeros,
though not by variation of arguments, but indirectly in the construction of a mollifier.

Levinson's method begins by writing the functional equation in the following form
\begin{equation}\label{A7}
\eta Y(s)X(s)L(s,f) = \eta X(s)G(s,f) + \overline{\eta}X(1-s)G(1-s,g)
\end{equation}
where $Y(s) = Y(1-s)$ is a simple function having only a few zeros.  For example we can arrange \eqref{A4} in the form
\begin{equation}\label{A8}
2\eta X(s)L(s,f) = \eta X(s)L(s,f) + \overline{\eta}X(1-s)L(1-s,g)
\end{equation}
which is a case of \eqref{A7} with $Y(s) = 2$ and $G(s,f) = L(s,f)$.  However, this simple arrangement yields poor results.
Of course, the $G$-function in \eqref{A7} is not defined uniquely.  Good results come out from \eqref{A7} with
$G(s,f)$ judiciously chosen.  We shall search for $G(s,f)$ in the class of Dirichlet series
\begin{equation}\label{A9}
G(s,f) = \sum_{1}^{\infty}a_{f}(n)n^{-s}.
\end{equation}
An advantage of such kind $G(s,f)$ is that one can control the variation of argument quite well on the lines of absolute
convergence.  Contrary, for example, the arrangement offered by the Riemann--Siegel formula \cite{Si} (a kind of approximate
functional equation) is not so good because the truncation of the relevant series is sharp at the point which depends
on the variable $s$.  Many similar functional equations can be developed which feature smooth decay transition, however
with coefficients depending on $s$, thus making it harder for mollification.

Excellent choices of $G(s,f)$ are proposed in \cite{C1}.  These are linear combinations of
$L(s,f)$, $L'(s,f)$, $L''(s,f),\ldots$.  For example one may take (the original one of Levinson)
\begin{equation}\label{A10}
G(s,f) = L(s,f) + \lambda L'(s,f)
\end{equation}
where $\lambda$ is a constant at our disposal.  Indeed \eqref{A7} holds for $G(s,f)$ given by \eqref{A10} with
\begin{equation}\label{A11}
Y(s) = 2 -\lambda\frac{X'}{X}(s)-\lambda\frac{X'}{X}(1-s).
\end{equation}
To check this, take the logarithmic derivative of \eqref{A4} and combine the resulting equation with \eqref{A8}.

The combination \eqref{A10} is particularly attractive for
\begin{equation}\label{A12}
\lambda = (\log{N})^{-1}
\end{equation}
where $N$ just exceeds the analytic conductor $q(|t|+3)^{d}$.  In this case, one can truncate the series for $L(s,f)$ and
$L'(s,f)$ at $n = N$ with very small error term.  Hence $G(s,f)$ is very well approximated by
\begin{equation}\label{A13}
G_{N}(s,f) = \sum_{n \leq N}\lambda_{f}(n)\Big(1 - \frac{\log{n}}{\log{N}}\Big)n^{-s}.
\end{equation}

Now we return to the general setting \eqref{A7}.  Observe that for $s$ on the line $\Re{s} = 1/2$ the equation
\eqref{A7} reads as
\begin{equation}\label{A14}
\eta Y(s)X(s)L(s,f) = 2\Re{\eta X(s)G(s,f)}.
\end{equation}
Hence $\Re{\eta X(s)G(s,f)} = 0$ if and only if $Y(s)L(s,f) = 0$.  In other words $s=1/2 + i\gamma$ is a critical
zero of $L(s,f)$ if and only if either $G(s,f) = 0$, or
\begin{equation}\label{A15}
G(s,f) \neq 0, \quad \arg{\eta X(s)G(s,f)} \equiv \pi/2 \pmod{\pi}
\end{equation}
except for a few zeros of $Y(s)$.  Suppose $Y(s)$ has at most $O(\log{qT})$ zeros on the segment
$\mathcal{C} = \{s = 1/2+it; |t|\leq T\}$.  Note that $Y(s)$ given by \eqref{A11} does satisfy this condition.
This can be verified by using Stirling's formula for $X'(s) / X(s)$.

Observe that for every change of $\pi$ in the argument of some function $f(z)$ it must be the case that $\Re{f(z)}$ has at
least one zero.  Hence it follows that
\begin{equation}\label{A16}
N_{0}(T,f) \geq \frac{1}{\pi}\Delta_{\mathcal{C}}\arg{X(s)G(s,f)} + O(\log{qT})
\end{equation}
where $\Delta_{\mathcal{C}}\arg$ stands for the variation of the argument as $s$ runs over the segment $\mathcal{C}$
from $1/2 - iT$ to $1/2+iT$ passing the zeros of $G(s,f)$ from the east side.  Actually the lower bound \eqref{A16} is for
the number $N^{'}_{0}(T,f)$ of simple zeros of $L(s,f)$ on $\mathcal{C}$, because if $s$ is a double zero then $G(s,f) = 0$
(see \eqref{A10}),

It is quick to show by Stirling's formula that
\begin{eqnarray*}
\frac{1}{\pi}\Delta_{\mathcal{C}}\arg{X(s)} &=& \frac{dT}{\pi}\log{\frac{T}{2\pi e}} + \frac{T}{\pi}\log{q} + O(1)\\
& = & N(T,f) + O(\log{qT}).
\end{eqnarray*}
Hence
\begin{equation}\label{A17}
N_{0}(T,f) \geq N(T,f) + \frac{1}{\pi}\Delta_{\mathcal{C}}\arg{G(s,f)} + O(\log{qT}).
\end{equation}

Next let $R$ be the closed rectangle whose left side is $\mathcal{C}$ and the right side is sufficiently far.
Let $\mathcal{R} = \partial R$ denote the boundary of $R$.  By standard techniques (for example see the proof of
Theorem 5.8 of \cite{IK}) one can show that the variation of argument of $G(s,f)$ on $\mathcal{R}\setminus\mathcal{C}$
is bounded by $O(\log{qT})$ so
$$\Delta_{\mathcal{C}}\arg{G(s,f)} = \Delta_{\mathcal{R}}\arg{G(s,f)} + O(\log{qT}).$$
Now
$$-\frac{1}{2\pi}\Delta_{\mathcal{R}}\arg{G(s,f)} = N_{R}(G), \quad \text{say},$$
is just the number of zeros of $G(s,f)$ inside the rectangle $R$ (the minus sign because of the clockwise orientation
of $\mathcal{R}$).  Hence
\begin{equation}\label{A18}
N_{0}(T,f) \geq N(T,f) - 2N_{R}(G) + O(\log{qT}).
\end{equation}

A few words of reflection are due at this moment.  First of all we came back in \eqref{A18} to a problem of counting zeros,
now those of $G(s,f)$ in the rectangle rather than those of $L(s,f)$ on the line.  Moreover we need an upper bound for
$N_{R}(G)$ to get a lower bound for $N_{0}(T,f)$.  The new task is definitely easier because it depends essentially on
estimates for the relevant analytic functions.  However one cannot guarantee success upfront.  There is a risk of losing a
large constant factor in the upper bound for $N_{R}(G)$ and the whole work is vein.

Clearly $N_{R}(G)$ can only increase if we replace $G(s,f)$ by
\begin{equation}\label{A19}
F(s,f) = G(s,f)M(s,f)
\end{equation}
where $M(s,f)$ is any regular function in the rectangle $R$.  This extra factor may add zeros, but hopefully not a lot.
On the other hand $M(s,f)$ is designed to dampen extra large values of $G(s,f)$ so the product $F(s,f)$ has more steady
behaviour than $G(s,f)$.  Consequently, counting zeros of $F(s,f)$ by classical methods of contour integration becomes
plausible.

Specifically we are going to apply the well-known formula of Littlewood \cite{Li}
\begin{equation}\label{A20}
\Re{\Big(\frac{1}{2\pi i}\int_{\partial D}\log{F(s)}ds\Big)} = \sum_{\rho \in D}\text{dist}(\rho).
\end{equation}
Here $\log{F(s)}$ is a continuous branch of logarithm,
$$\log{F(s)} = \log{|F(s)|} + i\arg{F(s)},$$
where the argument is defined by continuous variation going counter-clockwise.  This holds for a regular function $F(s)$
in a rectangle $D$, not vanishing on $\partial D$, where $\rho$ runs over the zeros of $F(s)$ and $\text{dist}(\rho)$ denotes
the distance of $\rho$ to the left side of $D$.

For our application we take $D$ somewhat wider that $R$ so the zeros in $R$ have an ample distance to the left side of $D$.
Specifically we expand $R$ by moving its left side at $\Re{s} = 1/2$ to $\Re{s} = \sigma$ with $\sigma <1/2$.  Then for
every $\rho \in R$ we have $\text{dist}(\rho) \geq 1/2 - \sigma$, so \eqref{A20} yields
\begin{eqnarray}\label{A21}\nonumber
(\frac{1}{2}-\sigma)N_{R}(G) & \leq & (\frac{1}{2}-\sigma)N_{R}(F)\\
& \leq & \Re{\Big(\frac{1}{2\pi i}\int_{\partial D}\log{F(s)}ds\Big)}.
\end{eqnarray}
The integration over the left side of $D$ yields exactly
\begin{equation}\label{A22}
\frac{1}{2\pi}\int_{-T}^{T}\log{|F(\sigma + it)|}dt.
\end{equation}
The contribution of the integration over the remaining parts of $\partial D$ can be estimated by $O(\log{qT})$.  This
requires some conditions on the mollifier.  Assume that $M(s)$ is given by a Dirichlet polynomial
\begin{equation}\label{A23}
M(s) = \sum_{m \leq X}c(m)m^{-s}
\end{equation}
of length $X$ (nothing to do with the function $X(s)$ in \eqref{A5}) and coefficients $c(1) = 1$, $c(m) \ll m$.
Assume $\log{X} \ll \log{qT}$.  Then
$$\log{M(s)} = \sum_{m=2}^{\infty}\alpha(m)m^{-s}$$
with $\alpha(m) \ll m^{2}$, so the series converges absolutely for $\Re{s} \geq 3$.  Hence
$$\frac{1}{2\pi i}\int_{4 - iT}^{4+iT}\log{M(s)}ds \ll \sum_{m=2}^{\infty}\frac{|\alpha(m)|}{m^{4}\log{m}} \ll 1.$$
Moreover $M(s) \ll X^{2}$ in $D$, so the real part of integrals over the horizontal segments (the integrals of
$\arg{F(\alpha+ iT)}$ and $\arg{F(\alpha-iT)}$) are bounded by $O(\log{TX})$.

Collecting these estimates we get by \eqref{A21}
\begin{equation}\label{A24}
(\frac{1}{2}-\sigma)N_{R}(G)  \leq  \frac{1}{2\pi}\int_{-T}^{T}\log{|F(\sigma + it)|}dt +O(\log{qT}).
\end{equation}
Finally by \eqref{A18} and \eqref{A22} we arrive at
\begin{proposition1}\label{propositionA}
Let $L(s,f)$ be an $L$-function of degree $d$ and conductor $q$ which satisfies the functional equation in the form
\eqref{A7} with $G(s,f)$ given by a Dirichlet series \eqref{A9}.  Let $M(s,f)$ be a Dirichlet polynomial of length
$X \ll (qT)^{A}$ given by \eqref{A23}.  Then
\begin{equation}\label{A25}
N_{0}(T,f) \geq N(T,f) - \frac{1}{\pi(\frac{1}{2}-\sigma)}\int_{-T}^{T}\log{|F(\sigma + it,f)|}dt +O(\log{qT})
\end{equation}
where $F(s,f) = G(s,f)M(s,f)$ and $0 < \sigma < 1/2$.  The implied constant in the error term $O(\log{qT})$
depends on the local parameters $\kappa_{1},\ldots,\kappa_{d}$.
\end{proposition1}

One can generalize Proposition \ref{propositionA} for a family of $L$-functions.  Suppose for every $f \in \mathcal{F}$ we
have $L(s,f)$ of the same degree $d$ and of various conductors $q$, but of the same order of magnitude, say
\begin{equation}\label{A26}
q \asymp Q.
\end{equation}
Denote
\begin{equation}\label{A27}
N(T,\mathcal{F}) = \sum_{f \in \mathcal{F}}c_{f}N(T,f)
\end{equation}
\begin{equation}\label{A28}
N_{0}(T,\mathcal{F}) = \sum_{f \in \mathcal{F}}c_{f}N_{0}(T,f),
\end{equation}
where $c_{f}$ are positive numbers with
\begin{equation}\label{A29}
\sum_{f \in \mathcal{F}}c_{f} = 1.
\end{equation}
Then Proposition \ref{propositionA} yields
\begin{equation}\label{A30}
N_{0}(T,\mathcal{F}) \geq N(T,\mathcal{F}) - \frac{2T}{\pi(\frac{1}{2}-\sigma)}\mathcal{J}(T,\mathcal{F}) +O(\log{QT})
\end{equation}
where
\begin{equation}\label{A31}
\mathcal{J}(T,\mathcal{F}) = \frac{1}{2T}\int_{-T}^{T}\sum_{f \in \mathcal{F}}c_{f}\log{|F(\sigma+it,f)|}dt.
\end{equation}

Let us introduce the so called analytic conductor of the family $\mathcal{F}$ by
\begin{equation}\label{A32}
N = QT^{d}.
\end{equation}
Then \eqref{A6} gives
\begin{equation}\label{A33}
N(T,f) = \frac{T}{\pi}(\log{N})\Big(1 + O\Big(\frac{1}{\log{T}}\Big)\Big)
\end{equation}
for every $f \in \mathcal{F}$.  Note that \eqref{A6} and \eqref{A33} are valuable results only for $T$ sufficiently large.
If $T$ is bounded we have no results, even if the conductor $q$ is large.  Hence $N(T,\mathcal{F})$ also satisfies
\eqref{A33} so \eqref{A30} implies
\begin{equation}\label{A34}
N_{0}(T,\mathcal{F}) \geq \Big(1 - \frac{2\mathcal{J}(T,\mathcal{F})}{(\frac{1}{2}-\sigma)\log{N}}
+ O\Big(\frac{1}{\log{T}}\Big)\Big)N(T,\mathcal{F}).
\end{equation}
In practice a good value for $\sigma$ is close to $1/2$, namely
\begin{equation}\label{A35}
\sigma = \frac{1}{2} - \frac{R}{\log{N}}
\end{equation}
where $R$ is a positive constant.  For this $\sigma$ the bound \eqref{A34} becomes
\begin{equation}\label{A36}
N_{0}(T,\mathcal{F}) \geq \Big(1 - \frac{2}{R}\mathcal{J}(T,\mathcal{F})
+ O\Big(\frac{1}{\log{T}}\Big)\Big)N(T,\mathcal{F}).
\end{equation}

Now the question is how to estimate $\mathcal{J}(T,\mathcal{F})$?  By the convexity of the logarithm function we get
\begin{equation}\label{A37}
\mathcal{J}(T,\mathcal{F}) \leq \log{\mathcal{K}(T,\mathcal{F})}
\end{equation}
where
\begin{equation}\label{A38}
\mathcal{K}(T,\mathcal{F}) = \frac{1}{2T}\int_{-T}^{T}\sum_{f \in \mathcal{F}}c_{f}|F(\sigma+it,f)|dt.
\end{equation}

There are various possibilities to estimate $\mathcal{K}(T,\mathcal{F})$.  Recall that $F(s,f)$ is given by a Dirichlet
series so there is a great deal of technology available to address the issue.  Particularly the technology is well developed
for handling the second power moments.  Therefore we apply the Cauchy--Schwarz inequality
\begin{equation}\label{A39}
\mathcal{K}(T,\mathcal{F}) \leq \mathcal{L}(T,\mathcal{F})^{1/2}
\end{equation}
and reduce the problem to estimation of
\begin{equation}\label{A40}
\mathcal{L}(T,\mathcal{F}) = \frac{1}{2T}\int_{-T}^{T}\sum_{f \in \mathcal{F}}c_{f}|F(\sigma+it,f)|^{2}dt.
\end{equation}
Applying the above inequalities to \eqref{A36} we arrive at
\begin{corollary1}\label{corollaryA}
Suppose the conditions of Proposition \ref{propositionA} hold for every $f$ in the family $\mathcal{F}$.  Let $\sigma$ be given
by \eqref{A35}.  Then
\begin{equation}\label{A41}
N_{0}(T,\mathcal{F}) \geq \Big(1 - \frac{1}{R}\log{\mathcal{L}(T,\mathcal{F})}
+ O\Big(\frac{1}{\log{T}}\Big)\Big)N(T,\mathcal{F}),
\end{equation}
where $\mathcal{L}(T,\mathcal{F})$ is the mean value of $|F(\sigma+it,f)|^{2}$ given in \eqref{A40} and
$F(s,f) = G(s,f)M(s,f)$.
\end{corollary1}

\begin{remark1}
The lower bound (\ref{A41}) remains true for $N^{'}_0(T,f)$ in place of $N_0(T,f)$.
\end{remark1}

\end{document}